\documentclass[12pt]{amsart}
\usepackage{latexsym}
\usepackage{enumitem}
\usepackage{graphicx}
\usepackage{subfig}
\usepackage{float}
\usepackage{relsize}
\usepackage[top=3.2cm,bottom=3.8cm,left=3cm,right=2cm]{geometry}
\usepackage{mathrsfs}
\usepackage{amssymb}
\usepackage{amsbsy}
\usepackage{amsmath}
\usepackage{CJK}
\usepackage{tikz}
\usepackage[colorlinks,linkcolor=blue,citecolor=blue,pagebackref]{hyperref}
\renewcommand {\thefootnote}{\fnsymbol{footnote}}

\newtheorem{theorem}{Theorem}
\newtheorem{thm}{Theorem}[section]
\newtheorem{lem}[thm]{Lemma}
\newtheorem{cor}[thm]{Corollary}

\newtheorem{rem}[thm]{Remark}
\newtheorem{problem}{Problem}

\theoremstyle{definition}
\newtheorem{de}[thm]{Definition}

\def\r{\mathbb{R}^n}
\def\s{\mathbb{S}^{n-1}}

\def\R{\mathcal{R}}

\def\conv{\text{\rm conv}\thinspace}

\allowdisplaybreaks[4]

\numberwithin{equation}{section}

\newcommand\nnfootnote[1]{%
	\begin{NoHyper}
		\renewcommand\thefootnote{}\footnote{#1}%
		\addtocounter{footnote}{-1}%
	\end{NoHyper}
}

\makeatletter

\SetLabelAlign{hang}{%
  #1%
  \aftergroup\adjustparshapeindent}
\newcommand*\adjustparshapeindent{%
  \@ifnextchar\egroup
    {\aftergroup\adjustparshapeindent}
    {\adjustparshapeindent@auxi}}
\newcommand*\adjustparshapeindent@auxi{%
  \unless\ifdim\wd\@tempboxa=\labelwidth
    \adjustparshapeindent@auxii
  \fi}
\newcommand*\adjustparshapeindent@auxii{%
  \dimen@ = \dimexpr\wd\@tempboxa-\labelwidth\relax
  \labelwidth = \wd\@tempboxa
  \advance\linewidth -\dimen@
  \advance\leftmargin \dimen@
  \advance\@totalleftmargin \dimen@
  \parshape \@ne \@totalleftmargin \linewidth}

\makeatother

\begin{document}
	
\begin{center}
	{\Large \bf Fixed and periodic points of the intersection body operators of lower orders}
\end{center}

\vskip 10pt

\begin{center}
	{\bf Cheng\ Lin\ \  \ \ \ \ \ Ge\ Xiong}\\~~ \\
		\small{School of Mathematical Sciences, Key Laboratory of Intelligent Computing and Applications (Ministry of Education), Tongji University, Shanghai, 200092, China}
\end{center}

\vskip 5pt

   \nnfootnote{E-mail addresses: 1. lcbruce@foxmail.com;\ 2. xiongge@tongji.edu.cn.}
\nnfootnote{Research of the authors was supported by NSFC No. 12271407.}

\begin{center}
	\begin{minipage}{15cm}
{\bf Abstract:}  For the intersection body
  operator of lower order $I_iK$ of a star body $K$ in $\r$, $i\in\{1, 2,\ldots, n-2\}$, we prove that $I_i^2K = cK$ iff $K$ is an origin-symmetric ball, and hence $I_iK = cK$ iff $K$ is an origin-symmetric ball. Combining the  recent breakthrough (case $i = n-1$) \cite{MSY}  of Milman, Shabelman and Yehudayoff,   slight modifications of two long-standing questions posed by R. Gardner \cite[Problems {\color{blue}8.6} and {\color{blue}8.7}]{1995} are completely solved.  As applications, we show that for the spherical Radon transform $\R$,  a
non-negative $\rho\in L^{\infty}(\s)$ satisfies $\R(\rho^i) = c\rho$ for some $c>0$ iff $\rho$ is constant. Also, the sharp Busemann intersection type \mbox{inequalities are established.}

		
		\end{minipage}
\end{center}
\begin{CJK*}{UTF8}{gbsn}
\vskip 20pt
\section{\bf Introduction}
\label{1}
\vskip 10pt

A Borel set $K$ in $\r$ is called \emph{star-shaped}, if $K = \{ru : r \in [0,\rho_K(u)],\thinspace u \in\s\}$ for some Borel function $\rho_K : \s \to \mathbb{R}_+$ called its radial function, where $\s$
denotes the Euclidean unit sphere in $\r$ and $\mathbb{R}_+=[0,\infty)$. The star-shaped set $K$ is called a \emph{star body}, if $\rho_K$ is positive and continuous. A star body $K$ in $\r$ is called a \emph{Lipschitz star body}, if $\rho_K$ is Lipschitz continuous.

The \emph{intersection body} $IK$ of a star body $K$ in $\r$ was introduced and studied by
Lutwak in \cite{Lutwak2}, who defined $IK$ as the star body given by
\begin{align*}
    \rho_{IK}(u)=V_{n-1}(K\cap u^{\perp})=\frac{1}{n-1}\int_{\s\cap u^{\perp}} \rho_K^{n-1}(\theta)\thinspace d\theta,\quad \forall\thinspace u\in \s.
\end{align*}
Throughout this article,  $V_k(\cdot)$ denotes the $k$-dimensional Hausdorff measure $\mathcal{H}^k(\cdot)$.

Intersection bodies play an essential role  in the dual Brunn-Minkowski theory and in Geometric Tomography, particularly in relation to
the solution of the celebrated Busemann-Petty problem. Please refer to, e.g., \cite{BP,Gardner, Gardner1, GKS, Koldobsky2, Lutwak2, Zhang3} and \cite[Chapter 8]{1995} for details.

In 1994,  Zhang \cite{Zhang1} generalized the notion of intersection body $IK$ to the \emph{intersection body of order $i$}, $I_i K$, $i\in\{1, 2,\ldots, n-1\}$,  defined as the star body given by
\begin{align*}
    \rho_{I_i K}(u)= \widetilde{V}_{i}(K\cap u^{\perp})=\frac{1}{n-1}\int_{\s\cap u^{\perp}} \rho^{i}_K(\theta)\thinspace d\theta,\quad \forall\thinspace u\in \s.
\end{align*}
Here, $\widetilde{V}_i(K\cap u^{\perp})$ is the \emph{$i$-th dual volume} of $K\cap u^{\perp}$ in $\mathbb{R}^{n-1}$. See Section \ref{2} for its definition. It is clear that $I_{n-1}K=IK$ and $I_i K$ is an intersection body of a star body. The generalized \emph{Funk section theorem} \cite[Theorem 7.2.6]{1995}, with $k=n-1$, shows that the map $I_i$, when restricted to the class of origin-symmetric star bodies, is \emph{injective}. For more information on $I_iK$,  refer to, e.g., \cite{ Hadwiger, Koldobsky1, RZ, TX, Zhang1}.

In the excellent book \cite{1995} authored by Gardner, the following problems are posed.
\begin{problem}[{\cite[Problem 8.6]{1995}}]
\label{8.6}
    Suppose that $1\leq i\leq n-1$. Which star bodies $K$ in $\r$ are such that $I_i^2 K$ is homothetic to $K$?
\end{problem}
\begin{problem}[{\cite[Problem 8.7]{1995}}]
\label{8.7}
    Suppose that $1\leq i\leq n-1$. Which star bodies $K$ in $\r$ are such that $I_i K$ is homothetic to $K$?
\end{problem}

 Grinberg and Zhang \cite[Corollary 9.8]{GZ} proved that when $n \geq  3$, if $I_1K = cK$ for some $c > 0$, then $K$ is an origin-symmetric ball. In 2011, Fish, Nazarov, Ryabogin and Zvavitch \cite[Theorem 1]{FNRZ} proved that if $K$ is a star body in $\r$ sufficiently close to  the Euclidean unit ball $B_n$ in the Banach-Mazur metric, then $I^m K\to B_n$ as $m\to\infty$. For such $K$, if $I^mK=cK$ for some $c > 0$ and an integer $m>0$, then $K$ is an origin-symmetric ellipsoid.

Very recently, Milman, Shabelman and Yehudayoff \cite{MSY} have made a \emph{breakthrough} on the above two long-standing questions (case $i=n-1$) and  completely solved the Open problem 12.8 posed by Lutwak \cite{Lutwak3}.

\begin{theorem}[{\cite{MSY}}]
\label{A}
    Let $K$ be a star body in $\r$ and $n\geq 3$. Then $I^2K = cK$ for some $c > 0$ if and only if $K$ is an origin-symmetric ellipsoid, and therefore $IK = cK$ for some $c > 0$ if and only if $K$ is an origin-symmetric ball.
\end{theorem}
Inspired by the novel idea and aided with the ingenious technique developed in \cite{MSY}, we solve the remaining cases of Problems 8.6 and 8.7  with  more \emph{natural} assumptions.

\begin{thm}
\label{periodic}
    Let $K$ be a star body in $\r$, $n\geq 3$ and $i\in\{1,2,\ldots,n-2\}$. Then $I_i^2K = cK$ for some $c > 0$ if and only if $K$ is an origin-symmetric ball, and therefore $I_iK = cK$ for some $c > 0$ if and only if $K$ is an origin-symmetric ball.
\end{thm}

\begin{rem}
    \label{smooth}
     Theorem \ref{periodic} actually holds under a more general assumption that $K$ is a star-shaped bounded Borel set in $\r$ satisfying $I_i^2K=cK$ or $I_iK = cK$ up to null-sets. Indeed, by Theorem \ref{regularity}, i.e., a simple adaption of the regularity analysis in \cite[Theorem A.1]{MSY}, it is possible to modify $K$ on a null-set so that either $K$ is the one-point set $\{o\}$ or else $K$ is an origin-symmetric ball.
\end{rem}

The above results admit a reformulation in  terms of \emph{non-linear} harmonic analysis. Suppose that $f$ is a Borel function on $\s$. The \emph{spherical Radon (or Funk) transform} $\R(f)$ of $f$ is defined by $\R(f)(u) = \int_{\s\cap u^{\perp}}f(\theta)\thinspace d\theta$, $\forall\thinspace u\in\s$. Hence,
\begin{align*}
    \rho_{I_iK}(u)=\frac{1}{n-1}\int_{\s\cap u^{\perp}}\rho_K^i(\theta)\thinspace d\theta=\frac{1}{n-1} \R(\rho_K^{i})(u),\quad\forall\thinspace u\in\s.
\end{align*}

 From Theorem \ref{periodic} and Remark \ref{smooth}, we obtain the following.
\begin{cor}
\label{functional}
    Let $\rho\in L^{\infty}(\s)$ be non-negative, $n \geq 3$ and $i\in\{1,2,\ldots,n-2\}$. Then as functions in $L^{\infty}(\s)$, $\R(\rho^{i})=c\rho$ for some $c>0$ if and only if $\rho$ is constant.
\end{cor}

A set $K\subseteq \r$ is called a \emph{convex body} in $\r$, if $K$ is a compact convex set with nonempty interior. The classical \emph{Busemann intersection inequality} \cite{Busemann} reads: If $K$ is a convex body  in $\r$ with origin in its interior and $n\geq 3$, then
\begin{align*}
    \frac{V_n(IK)}{V_n(K)^{n-1}}\leq  \frac{V_n(IB_n)}{V_n(B_n)^{n-1}}
\end{align*}
with equality if and only if $K$ is an origin-symmetric ellipsoid.

In this article, the following  Busemann intersection type inequalities for intersection
bodies of lower orders are established.
\begin{thm}
\label{isoperimetric}
    Let $K$ be a Lipschitz star body in $\r$, $n\geq 3$ and \mbox{$i\in \{1,2,\ldots,n-2\}$. Then}
    \begin{align*}
        \frac{\widetilde{V}_{i+1}(I_i K)}{\widetilde{V}_{i+1}(K)^i}\leq \frac{\widetilde{V}_{i+1}(I_i B_n)}{\widetilde{V}_{i+1}(B_n)^i}
    \end{align*}
     with equality if and only if $K$ is an origin-symmetric ball.
\end{thm}

To elucidate the strategy of our proof,  especially to stress which parts are new and which parts are essentially the same as in \cite{MSY},  we  try to summarize the proofs of Theorem \ref{A} and  Theorem \ref{periodic}, respectively.

Recall that to prove Theorem \ref{A}, Milman, Shabelman and Yehudayoff \cite[Proposition 6.2]{MSY} begin by identifying  $I^2K=cK$ as  the Euler-Lagrange equation of the functional $\mathcal{F}_c(K):=V_n(IK)-c(n-1)V_n(K),$
 and then characterize stationary points of $\mathcal{F}_c(K)$ under \emph{admissible radial perturbations} of star body $K$: the \emph{continuous Steiner symmetrization} $\{S^t_uK\}_{t \in [0,1]}$. In light of that $V_n(S^t_uK)=V_n(K)$, $t\in[0,1]$, they reduce the analysis of the equation $\frac{d}{dt}\big|_{t=0^+}\mathcal{F}_c(S^t_uK)=0$ to that of the equation
\begin{align}
\label{1.1}
    \frac{d}{dt}\Big|_{t=0^+}V_n(I(S^t_uK))=0.
\end{align}

Following \cite[Proposition 6.2]{MSY}, for each $i\in\{1,2,\ldots,n-2\}$,  we (Theorem \ref{Euler}) begin by identifying $I_i^2 K=cK$ as the Euler-Lagrange equation of  $\mathcal{F}_{c,i}(K):=\widetilde{V}_{i+1}(I_i K)-ci\widetilde{V}_{i+1}(K)$.
Then, we have to characterize stationary points of $\mathcal{F}_{c,i}(K)$ under our \emph{chosen} admissible radial perturbations  $\{K_t:= \langle (S^t_{u} \langle K^{\frac{i+1}{n}}\rangle)^{\frac{n}{i+1}}\rangle\}_{t\in[0,1]}$ of star body $K$, which is constructed through our \emph{newly} defined star body $\langle K^{\frac{i+1}{n}}\rangle$: the operation of taking a power of radial function $\rho_K$.  The advantage of this perturbation is that $\widetilde{V}_{i+1}(K_t)=\widetilde{V}_{i+1}(K)$, $t\in[0,1]$, so  we can
reduce the analysis of the equation $\frac{d}{dt}\big|_{t=0^+}\mathcal{F}_{c,i}(K_t)=0$ to that of the equation
\begin{align}
\label{1.2}
    \frac{d}{dt}\Big|_{t=0^+}\widetilde{V}_{i+1}(I_i(K_t))=0.
\end{align}

Second, to analyze \eqref{1.1}, by a novel application of the Blaschke-Petkantschin formula, Milman, Shabelman and Yehudayoff  \cite{MSY} reformulate  $V_n(IK)$ into
\begin{align*}
    \mathcal{I}_{0}(K):=\lim _{p \to (-1)^{+}} (p+1)(n-1)! \int_{(\mathbb{R}^{n})^n} V_n(\conv \{o,x_1,\ldots,x_n\})^{p} \mathsmaller{\prod}\limits_{j=1}^{n}1_{K}(x_j)\thinspace d x_{1} \cdots d x_{n}.
\end{align*}
 To calculate the limit involved in $\mathcal{I}_0(K)$, they further reformulate  $\mathcal{I}_0(K)$ into $\mathcal{I}_u(K)$, so that they can analyze
the behavior of $V_n(I(S^t_uK))$ by using the formula of $\mathcal{I}_{u}(S^t_uK)$.  Please refer to \cite[Theorems 1.9 and 1.12]{MSY} for details.

By contrast, to analyze \eqref{1.2},  using the  Blaschke-Petkantschin formula and our defined star body $\langle K^{\frac{i+1}{n}}\rangle$, we reformulate  $\widetilde{V}_{i+1}(I_i K)$ into
\begin{align*}
    \mathcal{V}_{i+1}(K):=b_{n,i}\int_{(\r)^{i+1}}\frac{\prod_{j=1}^{i+1}(|x_j|^{1-\frac{n}{i+1}}1_{\langle K^{\frac{i+1}{n}}\rangle}(x_j))}{V_{i+1}(\conv\{o,x_1,\ldots,x_{i+1}\})}\thinspace dx_1\cdots dx_{i+1},
\end{align*}
which involves the so-called radial weights $|x_j|^{1-\frac{n}{i+1}}$. To exploit our chosen admissible radial perturbations $\{ K_t=\langle (S^t_{u} \langle K^{\frac{i+1}{n}}\rangle)^{\frac{n}{i+1}}\rangle\}_{t\in[0,1]}$, we further reformulate $\mathcal{V}_{i+1}(K)$ into $\mathcal{V}_{i+1,u}(K)$, so that we can analyze the behavior of $\widetilde{V}_{i+1}(I_i(K_t))$ by using the formula of $\mathcal{V}_{i+1,u}(K_t)$. Please refer to Theorem \ref{reformulation} for details.

Finally, via geometric characterizations for equation $\frac{d}{dt}\big|_{t=0^+}\mathcal{I}_{u}(S^t_uK)=0$ \cite[Theorem 1.13]{MSY}, together with delicate arguments involving a key lemma on linear functions \cite[Lemma 1.14]{MSY} and a local form of the Bertrand-Brunn characterization of ellipsoids \cite[Theorem 7.8]{MSY}, the authors \cite{MSY} conclude that $K$ is an origin-symmetric \emph{ellipsoid}.

 By contrast, for case $i\in\{1,2,\ldots,n-2\}$, to derive that $K$ is an origin-symmetric \emph{ball}, we make full use of the \emph{geometry} of the integrand involved in $\mathcal{V}_{i+1,u}(K_t)$. Indeed, the integrand involved in $\mathcal{V}_{i+1,u}(K_t)$ is
 \begin{align*}
     V_{i+1}\big(R_{\boldsymbol{y}}(S^t_u\langle K^{\frac{i+1}{n}}\rangle)\cap [\Lambda_{u,\boldsymbol{y}}]_{\alpha}\cap H_{\boldsymbol{y}}(z)\big).
 \end{align*}
 Loosely speaking, it is the $(i+1)$-dimensional volume of intersection of the   Cartesian product of chords of $S^t_u\langle K^{\frac{i+1}{n}}\rangle$, the level set of  volume functional on random simplex, and the   Cartesian product of level sets of radial weights $|x_j|^{1-\frac{n}{i+1}}$.  Analyzing equation  $\frac{d}{dt}\big|_{t=0^+}\mathcal{V}_{i+1,u}(K_t)=0$, we conclude that $\langle K^{\frac{i+1}{n}}\rangle$ is symmetric with respect to $u^{\perp}$  for a.e. $u\in\s$, which leads to that $K$ is an origin-symmetric ball. \mbox{See Section \ref{3.3} for details.}

The article is organized as follows. In Section $\ref{2}$,  we collect some basic facts on dual mixed volumes and  the classical Blaschke-Petkantschin formulas.  Preliminary results like the \emph{regularity} of the spherical Radon transform, the continuous Steiner symmetrization on $u$-multi-graphical sets, and the admissible radial perturbations of star bodies, as achieved by  Milman, Shabelman and Yehudayoff \cite{MSY}, are also provided. These results are \emph{crucial} to this article. Proofs of Theorems \ref{periodic} and \ref{isoperimetric} are presented in Sections $\ref{3}$ and $\ref{4}$, respectively.

\vskip 10pt
\section{\bf Preliminaries}
\label{2}
\vskip 0pt

As usual,  write $|x|$ for the standard Euclidean norm of $x$ and $x \cdot y$   for the standard inner product in $x, y\in\r$, respectively. Let $[x, y]$ be the closed line segment with endpoints $x$ and $y$.  Denote
by $B_n(r)$ the Euclidean ball of radius $r$ in $\r$ centered at the origin $o$. The volume of $B_n$ is $\omega_n=\pi^{\frac{n}{2}}/\Gamma(1+\frac{n}{2})$.

For $u \in \mathbb{S}^{n-1}$, write $L_{u}=\operatorname{span}(\{u\})$ and let $L_{u}^{y}=y+L_{u}$ be the line through $y\in u^{\perp}$ in the direction $u$. Let $P_{E}$ be the orthogonal projection onto a linear subspace $E$ of $\r$.

The notion of \emph{$i$-th dual volume}, $i\in \mathbb{R}$, was originally defined by Lutwak \cite{Lutwak1}. For a star body $K$ in $\r$, its $i$-th dual volume $\widetilde{V}_i(K)$ is defined as $\widetilde{V}_i(K)=\frac{1}{n}\int_{\s}\rho_K^{i}(u)\thinspace du.$

For $A,B \subseteq \r$, their \emph{Minkowski sum} $A+B$ is the set $ \{a + b : a \in
A, b \in B\}$. Write $\conv A$ for the \emph{convex hull} of $A$, i.e., the smallest convex set containing $A$.

The Brunn-Minkowski inequality  reads: If $K,L$ are convex
bodies in $\r$ and $\lambda\in[0,1]$, then $V_n((1-\lambda)K+\lambda L)^{\frac{1}{n}}\geq (1-\lambda)V_n(K)^{\frac{1}{n}}+\lambda V_n(L)^{\frac{1}{n}},$ with equality if and only if $K$ and $L$ are homothetic. By the Brunn-Minkowski inequality, we obtain the following immediately.
\begin{lem}
\label{Brunn}
    If $K$ is a convex body and $L$ is a k-dimensional convex set in $\r$, then the function $g(x)=V_k(K\cap(L+x))^{\frac{1}{k}},\thinspace x\in\r$, is concave on its support.
\end{lem}
\begin{proof}
    For $x,y\in\r$ and $\lambda\in[0,1]$, we have
    \begin{align*}
        &\quad \enspace K\cap (L+(1-\lambda)x+\lambda y)=K\cap ((1-\lambda)(L+x)+\lambda(L+y))
        \\
        &\supseteq (1-\lambda)(K\cap(L+x))+\lambda (K\cap(L+y)).
    \end{align*}

    W.l.o.g., assume $L\subseteq \mathbb{R}^k$. Then $K\cap (L+x)\subseteq x+\mathbb{R}^k$ and $K\cap(L+y)\subseteq y+\mathbb{R}^k$. By   the translation invariance of volume and the Brunn-Minkowski inequality in $\mathbb{R}^k$, we obtain
    \begin{align*}
        &\quad\enspace g((1-\lambda)x+\lambda y)=V_k(K\cap (L+(1-\lambda)x+\lambda y))^{\frac{1}{k}}
        \\
        &\geq V_k((1-\lambda)(K\cap(L+x))+\lambda (K\cap(L+y)))^{\frac{1}{k}}
        \\
        &= V_k((1-\lambda)(K\cap(L+x)-x)+\lambda (K\cap(L+y)-y))^{\frac{1}{k}}
        \\
        &\geq (1-\lambda)V_k(K\cap(L+x)-x)^{\frac{1}{k}}+\lambda V_k(K\cap(L+y)-y)^{\frac{1}{k}}
        \\
        &=(1-\lambda)g(x)+\lambda g(y).
    \end{align*}
    This completes the proof.
\end{proof}

Write $G_{n,k}$ for the Grassmannian of $k$-dimensional linear subspaces of $\r$.  The following integral geometric identities, which are often referred to as the \emph{Blaschke-Petkantschin formulas}, are needed. See, e.g., \cite[Theorem 2.1]{DPP}, \cite[Lemmas 5.1]{Gardner3} and \cite[Theorem 7.2.1]{SW} for details.
\begin{thm}
\label{oushi}
    If $h$ is a non-negative  Borel function on $(\r)^q$ and  $1 \leq q \leq k \leq n$, then
\begin{align*}
    &\int_{(\r)^q} h(x_1,\ldots,x_q)\thinspace dx_1\cdots dx_q
    \\
    =&\enspace c_{n,k,q}\int_{G_{n,k}}\int_{E^q} h(x_1,\ldots,x_q)V_q(\conv\{o,x_1,\ldots,x_q\})^{n-k}\thinspace dx_1\cdots dx_q dE,
\end{align*}
where $c_{n,k,q}=\frac{\omega_{n-q+1}\cdots\omega_n}{\omega_{k-q+1}\cdots \omega_k}(q!)^{n-k}$ and $dE$ is the Haar probability measure on $G_{n,k}$.

\end{thm}
In particular, if $1\leq q\leq k=n-1$, then
\begin{align*}
    &\int_{(\r)^q} h(x_1,\ldots,x_q)\thinspace dx_1\cdots dx_q \nonumber
    \\
    =\enspace&c_{n,n-1,q}\int_{G_{n,n-1}}\int_{E^q} h(x_1,\ldots,x_q)V_q(\conv\{o,x_1,\ldots,x_q\})\thinspace dx_1\cdots dx_q dE \nonumber
    \\
    =\enspace&\frac{c_{n,n-1,q}}{n\omega_{n}}\int_{\s}\int_{(u^{\perp})^q} h(x_1,\ldots,x_q)V_q(\conv\{o,x_1,\ldots,x_q\})\thinspace dx_1\cdots dx_q du \nonumber
    \\
    =\enspace&\frac{q!}{n\omega_{n-q}} \int_{\s}\int_{(u^{\perp})^q} h(x_1,\ldots,x_q)V_q(\conv\{o,x_1,\ldots,x_q\})\thinspace dx_1\cdots dx_q du.
\end{align*}

So, for each non-negative Borel function $F$ on $(\r)^q$ with $1\leq q\leq n-1$, letting
\begin{align*}
    h(x_1,\ldots,x_q)=F(x_1,\ldots,x_q)V_q(\conv\{o,x_1,\ldots,x_q\})^{-1},
\end{align*}
  we obtain  the following identity
\begin{align}
    \label{formula}
    \int_{\s}\int_{(u^{\perp})^q} F(x_1,\ldots,x_q)\thinspace dx_1\cdots dx_q du=\frac{n\omega_{n-q}}{q!}\int_{(\r)^q} \frac{F(x_1,\ldots,x_q)}{V_q(\conv\{o,x_1,\ldots,x_q\})}\thinspace dx_1\cdots dx_q.
\end{align}

Let $f : \r \to [0,\infty]$.  The \emph{level set} $[f]_\alpha$ of $f$ at $\alpha \in [0,\infty]$ is defined  by
    \begin{align*}
        [f]_\alpha=\{x\in\r: f(x)\geq \alpha\}.
    \end{align*}
    $f$ is called \emph{quasi-concave}, if its level sets $[f]_\alpha$ are
convex for all $\alpha\in[0,\infty]$. Keep in mind that $[\cdot]_\alpha$ always denotes the level set at the height $\alpha$ throughout this article.

A functional $F : (\r)^q \to [0,\infty]$ is called \emph{Steiner concave}, if for
every $u\in\s$ and $\boldsymbol{y} = (y_1,\ldots,y_q ) \in (u^{\perp})^q$, the function $F_{u,\boldsymbol{y}} : \mathbb{R}^q \to [0,\infty]$ given by
\begin{align*}
    F_{u,\boldsymbol{y}}(s_1,\ldots,s_q)=F(y_1+s_1u,\ldots,y_q+s_qu),\quad (s_1,\ldots,s_q)\in \mathbb{R}^q,
\end{align*}
is even and quasi-concave.  Refer to the excellent survey \cite{PP2} for more information.

Interpreting $1/0$
as $\infty$. For later use, we introduce the functional
$$\Lambda(x_1,\ldots,x_q)=V_q(\conv\{o,x_1,\ldots,x_q\})^{-1},\quad (x_1,\ldots,x_q)\in (\r)^q, \enspace 1\leq q\leq n.$$
Therefore,
\begin{align*}
    \Lambda_{u,\boldsymbol{y}}(s_1,\ldots,s_q)=V_q(\conv\{o,y_1+s_1u,\ldots,y_q+s_qu\})^{-1}, \quad (s_1,\ldots,s_q)\in \mathbb{R}^q.
\end{align*}

With the help of the following Theorem \ref{half-convex}, i.e., \cite[Proposition 4.1]{PP2} by  Paouris and  Pivovarov, we prove that the functional $\Lambda$ is Steiner concave for $1\leq q\leq n$.

\begin{thm}
\label{half-convex}
    Let  $C$ be a compact convex set in
$\r\times \mathbb{R}^N$ and
$P_t(x,y)=x+(y\cdot t)u$,  $(x,y)\in  \r\times \mathbb{R}^N$ for $t\in\mathbb{R}^N$ and  $u\in\s$. Then for all integers $1\leq q\leq n$, the function $t\mapsto \mathbf{V}_q(P_tC)$, $t\in \mathbb{R}^N$, is convex, where $\mathbf{V}_q(P_tC)$ is the $q$-th intrinsic volume of $P_tC$.
\end{thm}

\begin{lem}
\label{convex}
      $\Lambda:(\r)^q\to [0,\infty]$ is Steiner concave for  all integers $1\leq q\leq n$.
\end{lem}
\begin{proof}
For $u\in\s$ and $\boldsymbol{y} = (y_1,\ldots,y_q ) \in (u^{\perp})^q$,  by the definition of level set, we have
\begin{align*}
    [\Lambda_{u,\boldsymbol{y}}]_\alpha&=\{(s_1,\ldots,s_q):V_q(\conv\{o,y_1+s_1u,\ldots,y_q+s_qu\})^{-1}\geq \alpha\}
        \\
    &=\{(s_1,\ldots,s_q):V_q(\conv\{o,y_1+s_1u,\ldots,y_q+s_qu\})\leq \frac{1}{\alpha}\}, \quad \forall \alpha\in[0,\infty].
\end{align*}
So, it suffices to show that the function
\begin{align*}
    \mathbb{R}^q\ni (s_1,\ldots,s_q) \mapsto V_q(\conv\{o,y_1+s_1u,\ldots,y_q+s_qu\})
\end{align*}
is even and convex.

First, in light of that $\conv\{o,y_1+s_1u,\ldots,y_q+s_qu\}$ and $\conv\{o,y_1-s_1u,\ldots,y_q-s_qu\}$ are symmetric with respect to $u^{\perp}$, it follows that
\begin{align*}
   V_q(\conv\{o,y_1+s_1u,\ldots,y_q+s_qu\})=V_q(\conv\{o,y_1-s_1u,\ldots,y_q-s_qu\}),
\end{align*}
which implies that the above function  is even.

 Second,  let $\{e_k\}_{k=1}^{q}$ be an orthonormal basis of $\mathbb{R}^q$. Putting $t=\sum_{k=1}^q s_ke_k$, $N=j=q$, $C=\conv\{(o,o),(y_1,e_1),\ldots, (y_q,e_q)\}\subseteq \r\times \mathbb{R}^q$, and $ P_t(y_k,e_k)=y_k+(e_k\cdot t)u=y_k+s_ku$ into Theorem \ref{half-convex},  it follows that
\begin{align*}
    \mathbb{R}^q\ni (s_1,\ldots,s_q) \mapsto V_q(\conv\{o,y_1+s_1u,\ldots,y_q+s_qu\})=V_q(P_tC)=\mathbf{V}_q(P_tC)
\end{align*}
is convex. The last equality holds since the  Hausdorff dimension of set $P_tC$ is at most $q$ (See, e.g., \cite[Section A.4]{1995}).  This completes the proof.
\end{proof}

\subsection{Spherical Radon transform and its regularity} \
\label{2.1}
\vskip 5pt
Let $f$ be a Borel function on $\s$. Its \emph{spherical Radon (or Funk) transform} $\R(f)$ is
\begin{align*}
    \mathcal{R}(f)(u)=\int_{\mathbb{S}^{n-1} \cap u^{\perp}} f(\theta) d \theta,\quad \forall\thinspace u\in\s.
\end{align*}
The spherical Radon transform $\R$ is self-adjoint, in the sense that
\begin{align}
    \label{symmetric}
    \int_{\mathbb{S}^{n-1}} \mathcal{R}(f) g \thinspace d u=\int_{\mathbb{S}^{n-1}} f \mathcal{R}(g) d u, \quad \forall\thinspace f, g \in L^{\infty}\left(\mathbb{S}^{n-1}\right).
\end{align}
Refer to \cite[Appendix C.2]{1995} for details.

Given $n \geq 3$ and a real parameter $s \geq 0$, let  $H^s(\s)$ denote the fractional Sobolev space. In particular, $L^{\infty}(\s)\subseteq L^2(\s)=H^0(\s)$.

By applying Lemma \ref{regular}, i.e.,  Proposition A.5 in \cite{MSY}, we reproduce the proof strategy of Theorem A.1 in \cite{MSY} to obtain a direct adaptation of it.
\begin{lem}[{\cite[Proposition A.5]{MSY}}]
\label{regular}
    If $f \in H^s(\s) \cap L^{\infty}(\s)$, then
 \begin{align*}
     \R(f^k) \in H^{s+\frac{n}{2}-1}(\s) \cap L^{\infty}(\s)
 \end{align*}
 for all integer $k \geq 1$.
\end{lem}
\begin{thm}[after Milman,  Shabelman and Yehudayoff \cite{MSY}]
\label{regularity}
  Let  $f \in L^{\infty}(\s)$ satisfy
 $\R(\R(f^k)^k)=cf$
    for some integers $n\geq 3$, $k\geq 1$ and some real number $c\neq 0$. Then  (possibly modifying f on a null-set) $f \in C^{\infty}(\s)$. In addition, if $f$ is non-negative, then either it is identically zero or else it is strictly positive.
\end{thm}
\begin{proof}
    Since  $\R(\R(f^k)^k)=cf$ for $f\in L^\infty(\s) \subseteq H^0(\s)$, by using Lemma $\ref{regular}$ twice it follows that $f \in H^{s+n-2}(\s) \cap L^{\infty}(\s)$; For $f\in H^{s+n-2}(\s) \cap L^{\infty}(\s)$, since  $\R(\R(f^k)^k)=cf$, by using Lemma $\ref{regular}$ twice we have $f \in H^{s+2(n-2)}(\s) \cap L^{\infty}(\s)$. Repeating the arguments $m$ times, it follows that $f \in H^{m(n-2)}(\s) \cap L^{\infty}(\s)$. By a standard application of the Sobolev-Morrey embedding theorem \cite[Theorem 6.3]{HR}, it follows that $f$ is $C^{\infty}$-smooth, up to modifying $f$ on a null-set.

In addition, assume $f\in C^{\infty}(\s)$ is non-negative and $f (u_0) = 0$ for some $u_0\in\s$. Let $g=\R(f^k)$. Then $g\in C^{\infty}(\s)$ is also non-negative and $\R(g^k)(u_0) = cf (u_0) = 0$, which implies that $g$ vanishes on $\s\cap u_0^{\perp}$. In turn, this implies that $f$ vanishes on $\s\cap u^{\perp}$ for all
$u \in \s\cap u_0^{\perp}$, and therefore $f$ is identically zero on $\s$. Consequently,  $f$ is strictly positive if $f$ is not identically zero.
\end{proof}

\subsection{Continuous Steiner symmetrization}\
\label{2.2}
\vskip 5pt

Continuous Steiner symmetrization (CSS) for graphical domains has its origins in the
work of P\'{o}lya and Szeg\"{o} \cite[Note B]{PS}. For the class of convex bodies, CSS is  a particular case of a shadow system \cite{RS, Shephard}, a well-established and extremely powerful tool, which has played a crucial role in solving a wide range of geometric extremization problems. See, e.g., \cite{CG1, CG2, CFPP, MR, MY, RS, Saroglou, Shephard}, for details.

Following Milman, Shabelman and Yehudayoff \cite[Section 4]{MSY}, this part is devoted to a brief introduction of the CSS on a \emph{$u$-multi-graphical} set. Recall that $L_{u}^{y}=y+L_{u}$ is the line through the point $y\in u^{\perp}$ in the direction $u$.
\begin{de}[$u$-multi-graphical set]
\label{multi}
    Given $u \in \mathbb{S}^{n-1}$, a compact set $K$ in $\mathbb{R}^{n}$ is called $u$-multi-graphical, if there exist disjoint open sets $\Omega_{1}, \Omega_{2}, \ldots \subseteq P_{u^{\perp}} K$ and two sequences of continuous functions
$$
f_{i}, g_{i}: \mathsmaller{\bigcup}\limits_{m=i}^{\infty} \Omega_{m} \rightarrow \mathbb{R}
$$
such that the following properties hold:
\begin{enumerate}[left=0pt,  align=hang]
\item Denoting $\Omega_{\infty}:=\cup_{m} \Omega_{m}$, we have $\mathcal{H}^{n-1}\left(P_{u^{\perp}} K \backslash \Omega_{\infty}\right)=0$.
  \item $f_{1}<g_{1}<f_{2}<g_{2}<\cdots<f_{m}<g_{m}$ on $\Omega_{m}$.
\item For all $y \in \Omega_{m}$, $K \cap L_{u}^{y}=y+u \cup_{i=1}^{m}\left[f_{i}(y), g_{i}(y)\right]$.
\end{enumerate}
\end{de}

Following Milman, Shabelman and Yehudayoff \cite[Section 4]{MSY}, the CSS is constructed in two steps. First,  define $S^{t} J$ where $J \subseteq \mathbb{R}$ is a finite disjoint union of closed intervals $J=\cup_{i=1}^{m}\left[c_{i}-\ell_{i}, c_{i}+\ell_{i}\right]$ ($\ell_{i}>0$). The idea, going back to the work of Rogers \cite{Rogers} and Brascamp-Lieb-Luttinger \cite{BLL}, is as follows. Each interval $\left[c_{i}-\ell_{i}, c_{i}+\ell_{i}\right]$ is moved independently towards the origin at a constant speed of $-c_{i}$ until the first time $\tau \in(0,1)$ at which two intervals touch (if there is only one interval set $\tau=1$). Thus,
\begin{align*}
    S^{t} J=\cup_{i=1}^{m}\left((1-t) c_{i}+\left[-\ell_{i}, \ell_{i}\right]\right), \quad t \in[0, \tau].
\end{align*}
If $\tau<1$, this means that at time $\tau$ the number of intervals $m^{\prime}$ in $S^{\tau} J=\cup_{i=1}^{m^{\prime}}[c_{i}^{\prime}-\ell_{i}^{\prime}, c_{i}^{\prime}+\ell_{i}^{\prime}]$ has decreased, and we recursively set
\begin{align*}
    S^{t} J=S^{\frac{t-\tau}{1-\tau}}\left(S^{\tau} J\right), \quad t \in[\tau, 1].
\end{align*}

Second, this one-dimensional construction is extended fiberwise to the setting of a $u$-multi-graphical compact set $K$ as follows. Let $\Omega_{\infty}$ be the open
subset of $P_{u^{\perp}}K$ given by Definition $\ref{multi}$. The \emph{continuous Steiner symmetrization} $\{S^t_uK\}_{t\in[0,1]}$ of $K$ is defined by
    $$
\mathring{S}^t_uK:=\mathsmaller{\bigcup}\limits_{y \in \Omega_{\infty}} S^{t}\left(K \cap L_{u}^{y}\right)\quad\text{and}\quad S_{u}^{t} K:=\operatorname{cl}\big(\mathring{S}_{u}^{t} K\big),\quad \forall \thinspace t\in[0,1] .
$$

It is worth mentioning that $\{S_{u}^{t} K\}_{t\in[0,1]}$ does not depend on the particular choice of $\Omega_{\infty}$ in its $u$-multi-graphical representation. See \cite[Proposition 4.8]{MSY} for the proof. Now we collect some fundamental properties on CSS, established in \cite[Corollary 4.6, Corollary 4.10, Corollary 4.12, Proposition 5.8 and Corollary 5.9]{MSY}.

\begin{lem}
\label{Continuous Steiner symmetrization}
    Let $K$ be a $u$-multi-graphical compact set with $u\in \s$, and let $\Omega_{\infty}$ be the open subset of $P_{u^{\perp}}K$ given by Definition $\ref{multi}$. The following properties hold.
    \begin{enumerate}[left=0pt, label={\rm(\arabic*)}, align=hang]
        \item If  $y_k \to y \in \Omega_{\infty}$, then  $S^t(K\cap L^{y_k}_u) \to S^t(K\cap L^{y}_u)$ in the Hausdorff metric for $t\in[0,1]$.
        \item For $y \in \Omega_{\infty}$ and $t \in[0,1], S_{u}^{t} K \cap L_{u}^{y}=\mathring{S}_{u}^{t} K \cap L_{u}^{y}$.
        \item For $t \in [0, 1]$, $V_n(S^t_uK)=V_n(\mathring{S}^t_uK)=V_n(K)$.
        \item If $K$ is a Lipschitz star body, then $S^t_u K$ is a Lipschitz star body for $t\in[0,1]$.
        \item If $K$ is a Lipschitz star body, then $S^0_u K=K$.
    \end{enumerate}
\end{lem}
The following multi-graphical property of Lipschitz star bodies is included in Theorem 5.7 in \cite{MSY}, shown
by Lin and Xi \cite[Lemma 2.2, Sect. 3 and Theorem 4.1]{Lin2}.
\begin{thm}
\label{almost direction}
    If $K$ is a Lipschitz star body in $\mathbb{R}^{n}$, then there exists a Lebesgue measurable set $U(K) \subseteq \mathbb{S}^{n-1}$ of full measure such that  $K$ is $u$-multi-graphical for all $u \in U(K)$.

\end{thm}

\subsection{Admissible radial perturbations}\
\label{2.3}
\vskip5pt

For case $i\in\{1,2,\ldots,n-1\}$, we identify $I_i^2 K=cK$ as the Euler-Lagrange \mbox{equation for}
\begin{align*}
        \mathcal{F}_{c,i}(K):=\widetilde{V}_{i+1}(I_i K)-ci\widetilde{V}_{i+1}(K).
\end{align*}
For this aim,  we first make some preparations.

Given a function $f: J \rightarrow \mathbb{R}$ on an interval $J$ and $a\in J$, denote its \emph{lower right-hand derivative} at $t=a$ by $\underline{\frac{d}{dt}}f(t)\big|_{a^+}=\liminf_{t\to a^+}\frac{f(t)-f(a)}{t-a}$. If $f(t)$ is differentiable from the right at $t=a$,  denote its \emph{right-derivative} at $t=a$ by $\left.\frac{d}{d t} f(t)\right|_{a^{+}}$.

\begin{de}[{\cite[Definition 6.1]{MSY}}]
    Let $K$ be a star body in $\r$. A family
of star-shaped sets $\{K_t\}_{t\in[0,1]}$ is called an \emph{admissible radial perturbation}, if $K_0 = K$
and $\{[0, 1]\ni t\mapsto \rho_{K_t}(u)\}_{u\in\s}$ are a.e. equi-differentiable at $t = 0^+$ in the following sense.
\begin{enumerate}[left=0pt,  align=hang]
    \item  For almost every $u\in\s$, the following limit exists:
    \begin{align}
    \label{existence}
         \frac{d\rho_{K_t}(u)}{dt}\Big|_{0^+}=\lim_{t\to 0^+}\frac{\rho_{K_t}(u)-\rho_K(u)}{t}.
    \end{align}
    \item There exist $M > 0$ and $t_0\in(0,1]$ such that for almost every $u\in\s$,
    \begin{align}
    \label{bounded}
     \sup_{t\in(0,t_0]}\frac{|\rho_{K_t}(u)-\rho_K(u)|}{t}\leq M.
    \end{align}
\end{enumerate}
\end{de}

The following theorem is a straightforward adaptation of Proposition 6.2 in \cite{MSY}.

\begin{thm}[after Milman,  Shabelman  and Yehudayoff \cite{MSY}]
\label{Euler}
    Let $\{K_t\}_{t\in[0,1]}$ be an admissible radial perturbation of a star body $K$ in $\r$ and $i\in\{1,2,\ldots,n-1\}$. Then, denoting $f (u) := \frac{d\rho_{K_t}(u)}{dt}\big|_{0^+}$, the following derivatives exist and are given by
    \begin{align*}
        \frac{d\widetilde{V}_{i+1}(K_t)}{dt}\Big|_{0^+}&=\frac{i+1}{n}\int_{\s}\rho^i_K(u)f(u)du,
        \\
        \frac{d\widetilde{V}_{i+1}(I_i K_t)}{dt}\Big|_{0^+}&=\frac{i(i+1)}{n}\int_{\s}\rho_{I_i^2 K}(u) \rho^{i-1}_K(u)f(u)du.
    \end{align*}

    Consequently, $I_i^2K=c K$ if and only if $K$ is a stationary point of the functional $\mathcal{F}_{c,i}(K)$,
    meaning that $\frac{d\mathcal{F}_{c,i}(K_t)}{dt}\big|_{0^+}=0$ for any admissible radial perturbation $\{K_t\}_{t\in[0,1]}$.
\end{thm}

\begin{proof}
    By $\eqref{existence}$, the derivative $f (u) = \frac{d\rho_{K_t}(u)}{dt}\big|_{0^+}$ exists for almost every $u\in\s$ and thus is  Lebesgue measurable. Moreover, $f\in L^{\infty}(\s)$ by $\eqref{bounded}$.

    Since $K \subseteq B_n(r)$ for some $r>0$, $\eqref{bounded}$ implies that $\sup_{t\in(0,t_0]} \rho_{K_t}(u) \leq M + r$ for a.e. $u\in\s$. So, invoking $\eqref{bounded}$ again, we see that for all $m\geq 1$ and a.e. $u\in\s$,
    \begin{align*}
        \sup_{t\in(0,t_0]}\frac{|\rho_{K_t}^m(u)-\rho_K^m(u)|}{t}\leq C_{r,m,M}
    \end{align*}
    for some $C_{r,m,M} > 0$. Thus, by the Lebesgue dominant convergence theorem, we have
    \begin{align*}
        \frac{d\widetilde{V}_{i+1}(K_t)}{dt}\Big|_{0^+}=\frac{d}{dt}\Big(\frac{1}{n}\int_{\s}\rho^{i+1}_{K_t}(u)\thinspace du\Big)\Big|_{0^+}=\frac{i+1}{n}\int_{\s}\rho^i_K(u)f(u)du.
    \end{align*}

    Similarly, by $\eqref{symmetric}$, we have
    \begin{align*}
        \frac{d\widetilde{V}_{i+1}(I_i K_t)}{dt}\Big|_{0^+}&=\frac{d}{dt}\Big(\frac{1}{n}\int_{\s}\big(\frac{1}{n-1}\int_{\s\cap u^{\perp}}\rho^i_{K_t}(\theta)\thinspace d\theta\big)^{i+1}du\Big)\Big|_{0^+}
        \\
        &=\frac{i+1}{n}\int_{\s}\rho^i_{I_i K}(u)\big(\frac{i}{n-1}\int_{\s\cap u^{\perp}}\rho^{i-1}_{K}(\theta)f(\theta)\thinspace d\theta\big)\thinspace du
        \\
        &=\frac{i(i+1)}{(n-1)n}\int_{\s}\rho^i_{I_i K}(u)\R(\rho^{i-1}_{K}f)(u)\thinspace du
        \\
        &=\frac{i(i+1)}{(n-1)n}\int_{\s}\R(\rho_{I_i K}^i)(u)\rho^{i-1}_K(u)f(u)\thinspace du\\
        &=\frac{i(i+1)}{n}\int_{\s}\rho_{I_i^2 K}(u) \rho^{i-1}_K(u)f(u)\thinspace du.
    \end{align*}

    Thus,
    \begin{align*}
        \frac{d\mathcal{F}_{c,i}(K_t)}{dt}\Big|_{0^+}=\frac{i(i+1)}{n}\int_{\s}\rho^{i-1}_K(u)f(u)\big(\rho_{I_i^2 K}(u)-c\rho_K(u)\big)\thinspace du,
    \end{align*}
     which implies that $\frac{d\mathcal{F}_{c,i}(K_t)}{dt}\big|_{0^+}=0$ if $I_i^2K = cK$.

     Assume $I_i^2K \neq cK$. Define $\{K_t\}_{t\in[0,1]}$  by $\rho_{K_t}=\rho_K+\varepsilon t (\rho_{I_i^2 K}-c\rho_K)$ for a sufficiently small $\varepsilon>0$. Then $\{K_t\}_{t\in[0,1]}$ is an admissible radial perturbation and $f=\varepsilon  (\rho_{I_i^2 K}-c\rho_K)$, which results in that $\frac{d\mathcal{F}_{c,i}(K_t)}{dt}\big|_{0^+}>0$. It is a contradiction.
\end{proof}

By Theorem \ref{Euler},  observe that if $I_i^2K=c K$ and
$\frac{d\widetilde{V}_{i+1}(K_t)}{dt}\big|_{0^+}=0$, then $\frac{d\widetilde{V}_{i+1}(I_i K_t)}{dt}\big|_{0^+}=0$.

The following Lemma, which was  established by   Milman, Shabelman and Yehudayoff \cite[Proposition 6.3]{MSY}, is crucial and bridges CSS and admissible radial perturbations.

\begin{lem}
\label{admissible}
 Let $K$ be a Lipschitz star body in $\r$ and $u \in U(K)$, where $U(K) \subseteq \s$ is given by Theorem $\ref{almost direction}$. Then $\{ S^t_u K\}_{t\in[0,1]}$ is an
admissible radial perturbation of $K$.
\end{lem}

\vskip 20pt
\section{\bf The proof of main results}
\label{3}
\vskip 10pt

Along the clue of Theorem \ref{Euler}, to character the star body $K$ satisfying the equation $I_i^2K=cK$ for $i\in\{1,2,\ldots,n-2\}$,  we need to \emph{construct} an admissible radial perturbation $\{K_t\}_{t\in[0,1]}$ of $K$, and then to analyze  the functional $\mathcal{F}_{c,i}(K_t)=\widetilde{V}_{i+1}(I_i K_t)-ci\widetilde{V}_{i+1}(K_t).$

\subsection{New admissible radial perturbations for  Lipschitz star body \texorpdfstring{$K$}{}}\
\label{3.1}
\vskip 5pt

  Unlike the classical volume functional $V_n$ (since $\widetilde{V}_{n}=V_n$),  for $i\in\{1,2,\ldots,n-2\}$ and the continuous Steiner symmetrization $\{S^t_u K\}_{t \in [0,1]}$,  $\widetilde{V}_{i+1}(S^t_u K)$ usually  is not equal to $\widetilde{V}_{i+1}(K)$, even if $K$ is  convex.

 Indeed, let $K$ be a convex body in $\r$ with the origin in its interior and $u\in\s$. Assume $K$ is not symmetric with respect to $u^{\perp}$. Then $K$ is a $u$-multi-graphical set and
\begin{align*}
    &\quad\enspace\widetilde{V}_{i+1}(S^t_u K)=\frac{1}{n}\int_{\s}\rho_{S^t_u K}^{i+1}(v)\thinspace dv
   =\frac{i+1}{n}\int_{S^t_u K}|x|^{i+1-n}\thinspace dx
    \\
    &=\frac{i+1}{n}\int_{P_{u^{\perp}}K}\int_{S^t_u K\cap L^y_u}|y+su|^{i+1-n}\thinspace dsdy=\frac{i+1}{n}\int_{\Omega_{\infty}}\int_{S^t(K\cap L_u^y)}|y+su|^{i+1-n}\thinspace dsdy,
\end{align*}
where $\Omega_{\infty}$ is the open subset of $P_{u^{\perp}}K$ given by Definition $\ref{multi}$. Since $i+1-n< 0$, it follows that $\widetilde{V}_{i+1}(S^t_u K)$ is strictly increasing in $t\in[0,1]$ by the last equality.

Let $K$ be a star body in $\r$ and $q\in \mathbb{R}$. The \emph{power-body} $\langle K^{q}\rangle$ of $K$ is defined as the star body in $\r$ by
\begin{align*}
    \rho_{\langle K^{q}\rangle}(u)=\rho_K^{q}(u),\quad\forall\thinspace u\in\s.
\end{align*}
In the following, we verify that the family $\{K_t := \langle (S^t_{u} \langle K^{\frac{i+1}{n}}\rangle)^{\frac{n}{i+1}}\rangle\}_{t\in[0,1]}$ is an admissible radial perturbation, which satisfies $\widetilde{V}_{i+1}(K_t)=\widetilde{V}_{i+1}(K)$, $t\in[0,1]$.

\begin{lem}
\label{again admissible}
    Let $K$ be a Lipschitz star body in $\r$,   $i\in\{1,\ldots,n-2\}$ and  $u \in U(\langle K^{\frac{i+1}{n}}\rangle)$, where $U(\langle K^{\frac{i+1}{n}}\rangle) \subseteq \s$ is given by Theorem $\ref{almost direction}$. Then the family
    \begin{align*}
        \{K_t := \langle (S^t_{u} \langle K^{\frac{i+1}{n}}\rangle)^{\frac{n}{i+1}}\rangle\}_{t\in[0,1]}
    \end{align*}
    is an admissible radial perturbation of $K$, which satisfies $\widetilde{V}_{i+1}(K_t)=\widetilde{V}_{i+1}(K)$, $t\in[0,1]$.

    In addition, if $I_i^2K=c K$ for some $c>0$, then $\frac{d\widetilde{V}_{i+1}(I_i K_t)}{dt}\big|_{0^+}=0$.

\end{lem}
\begin{proof}
    Since $K$ is a Lipschitz star body, it follows that  $\langle K^{\frac{i+1}{n}}\rangle$ is also a Lipschitz star body, and hence $\{S^t_u \langle K^{\frac{i+1}{n}}\rangle\}_{t\in[0,1]}$ is an admissible radial perturbation of $\langle K^{\frac{i+1}{n}}\rangle$ by Lemma $\ref{admissible}$. Combining the fact that $x\mapsto x^{\frac{n}{i+1}}$ is continuously differentiable on $\mathbb{R}_+$, it follows that $\{K_t\}_{t\in[0,1]}$ is an admissible radial perturbation.

    By direct computation and Lemma $\ref{Continuous Steiner symmetrization}$ $(3)$, for each $t\in[0,1]$ we have
    \begin{align*}
    \widetilde{V}_{i+1}(K_t)=\frac{1}{n}\int_{\s}\rho^{i+1}_{K_t}(u)\thinspace du=\frac{1}{n}\int_{\s}\rho_{S^t_u \langle K^{\frac{i+1}{n}}\rangle}^n(u)\thinspace du=V_n(S^t_u \langle K^{\frac{i+1}{n}}\rangle)=V_n(\langle K^{\frac{i+1}{n}}\rangle),
\end{align*}
which yields that $\widetilde{V}_{i+1}(K_t)=V_n(\langle K^{\frac{i+1}{n}}\rangle)=\widetilde{V}_{i+1}(K)$, $t\in[0,1]$. In particular,
\begin{align*}
    \frac{d\widetilde{V}_{i+1}(K_t)}{dt}\Big|_{0^+}=\frac{d\widetilde{V}_{i+1}(K)}{dt}\Big|_{0^+}=0.
\end{align*}

In addition, if $I_i^2K=c K$ for some $c>0$,  by Theorem $\ref{Euler}$ it follows that
\begin{align*}
    \frac{d\widetilde{V}_{i+1}(I_i K_t)}{dt}\Big|_{0^+}=0.
\end{align*}
This completes the proof.
\end{proof}

\subsection{New integral formulas for \texorpdfstring{$\widetilde{V}_{i+1}(I_iK)$}{}}\
\label{3.2}
\vskip 5pt

To analyze the behavior of  $\widetilde{V}_{i+1}(I_i K_t)$ of star body $K$, $i\in\{1,2,\ldots,n-2\}$, we derive two new integral formulas for $\widetilde{V}_{i+1}(I_iK)$,  which remain novel even if $K$ is convex.

     For $u\in \s$ and $\boldsymbol{y}=(y_1,\ldots,y_{i+1})\in (u^{\perp})^{i+1}$,  define
    \begin{align*}
        R_{\boldsymbol{y}}(K)=\{(s_1,\ldots,s_{i+1})\in \mathbb{R}^{i+1}: y_j+s_ju\in K,\thinspace j=1,2,\ldots,i+1\}
    \end{align*}
    and
    \begin{align*}
        H_{\boldsymbol{y}}(z)=[\varphi_{y_1}]_{z_1}\times[\varphi_{y_2}]_{z_2}\times\cdots\times[\varphi_{y_{i+1}}]_{z_{i+1}}\subseteq \mathbb{R}^{i+1},
    \end{align*}
    where $\varphi_{y_j}(t)=|y_j+tu|^{1-\frac{n}{i+1}}$, $ t\in\mathbb{R}$ and $z=(z_1,z_2,\ldots,z_{i+1})\in (\mathbb{R}_+)^{i+1}$. Note that  $[\varphi_{y_j}]_{z_j}$ is the level set of $\varphi_{y_j}$ at the height $z_j$. Keep in mind that  $y_j\in u^{\perp}$ and $z_j\in \mathbb{R}_+$.

    Recall that in Section \ref{2}, for $q=i+1$ with $i\in\{1,2,\ldots,n-2\}$,  $\Lambda: (\r)^{i+1}\to [0,\infty]$ is
    \begin{align*}
        \Lambda(x_1,\ldots,x_{i+1})=V_{i+1}(\conv\{o,x_1,\ldots,x_{i+1}\})^{-1},\quad (x_1,\ldots,x_{i+1})\in (\r)^{i+1}.
    \end{align*}
    Using the functional $\Lambda$, we define
    \begin{align*}
        \mathcal{V}_{i+1}(K):=b_{n,i}\int_{(\r)^{i+1}}\Lambda(x_1,\ldots,x_{i+1})\mathsmaller{\prod}\limits_{j=1}^{i+1}(|x_j|^{1-\frac{n}{i+1}}1_{\langle K^{\frac{i+1}{n}}\rangle}(x_j))\thinspace dx_1\cdots dx_{i+1},
    \end{align*}
    and
    \begin{align*}
        \mathcal{V}_{i+1,u}(K):=b_{n,i}\int_{(P_{u^{\perp}}\langle K^{\frac{i+1}{n}}\rangle)^{i+1}}\int_{\mathbb{R}_+}\int_{(\mathbb{R}_+)^{i+1}}V_{i+1}\big(R_{\boldsymbol{y}}(\langle K^{\frac{i+1}{n}}\rangle)\cap [\Lambda_{u,\boldsymbol{y}}]_{\alpha}\cap H_{\boldsymbol{y}}(z)\big)\thinspace dz d\alpha d\boldsymbol{y}.
    \end{align*}
    Here, $b_{n,i}=\frac{\omega_{n-i-1}}{(i+1)!}\big(\frac{in}{(i+1)(n-1)}\big)^{i+1}$ and recall  that $\Lambda_{u,\boldsymbol{y}}:\mathbb{R}^{i+1}\to[0,\infty]$ is defined by
 \begin{align*}
     \Lambda_{u,\boldsymbol{y}}(s_1,\ldots,s_{i+1})=V_{i+1}(\conv\{o,y_1+s_1u,\ldots,y_{i+1}+s_{i+1}u\})^{-1}, \quad (s_1,\ldots,s_{i+1})\in \mathbb{R}^{i+1}.
 \end{align*}

\begin{thm}
    \label{reformulation}
    If $K$ is a star body in $\r$, then $\widetilde{V}_{i+1}(I_iK)=\mathcal{V}_{i+1}(K)=\mathcal{V}_{i+1,u}(K).$
\end{thm}
\begin{proof}
     Putting $F(x_1,\ldots, x_{i+1})=\Pi_{j=1}^{i+1}(|x_j|^{1-\frac{n}{i+1}}1_{\langle K^{\frac{i+1}{n}}\rangle}(x_j))$ into equality $\eqref{formula}$, we have
    \begin{align*}
        &\enspace\quad\widetilde{V}_{i+1}(I_iK)=\frac{1}{n}\int_{\s}\big(\frac{1}{n-1}\int_{\s\cap u^{\perp}}\rho^{i}_{K}(\theta)\thinspace d\theta\big)^{i+1}\thinspace d    u
    \\
    &=\frac{1}{n}\int_{\s}\big(\frac{1}{n-1}\int_{\s\cap u^{\perp}}\rho^{\frac{in}{i+1}}_{ \langle K^{\frac{i+1}{n}}\rangle}(\theta)\thinspace d\theta\big)^{i+1}du
    \\
    &=\frac{1}{n}\int_{\s}\big(\frac{in}{(i+1)(n-1)}\int_{\s\cap u^{\perp}}\int_0^{\rho_{\langle K^{\frac{i+1}{n}}\rangle}(\theta)}r^{\frac{in}{i+1}-1}\thinspace drd\theta\big)^{i+1}\thinspace du
        \\
        &=\frac{1}{n}\big(\frac{in}{(i+1)(n-1)}\big)^{i+1}\int_{\s}\big(\int_{\langle K^{\frac{i+1}{n}}\rangle\cap u^{\perp}}|x|^{1-\frac{n}{i+1}}\thinspace dx\big)^{i+1}\thinspace du
        \\
        &=\frac{1}{n}\big(\frac{in}{(i+1)(n-1)}\big)^{i+1}\int_{\s}\int_{(u^{\perp})^{i+1}}\mathsmaller{\prod}\limits_{j=1}^{i+1}(|x_j|^{1-\frac{n}{i+1}}1_{\langle K^{\frac{i+1}{n}}\rangle}(x_j))\thinspace dx_1\cdots dx_{i+1}du
        \\
    &=\frac{1}{n}\big(\frac{in}{(i+1)(n-1)}\big)^{i+1}\frac{n\omega_{n-i-1}}{(i+1)!}\int_{(\r)^{i+1}}\frac{\prod_{j=1}^{i+1}(|x_j|^{1-\frac{n}{i+1}}1_{\langle K^{\frac{i+1}{n}}\rangle}(x_j))}{V_{i+1}(\conv\{o,x_1,\ldots,x_{i+1}\})}\thinspace dx_1\cdots dx_{i+1}
    \\
    &=b_{n,i}\int_{(\r)^{i+1}}\Lambda(x_1,\ldots,x_{i+1})\mathsmaller{\prod}\limits_{j=1}^{i+1}(|x_j|^{1-\frac{n}{i+1}}1_{\langle K^{\frac{i+1}{n}}\rangle}(x_j))\thinspace dx_1\cdots dx_{i+1}
    =\mathcal{V}_{i+1}(K).
\end{align*}

   Furthermore, by the Fubini theorem, we have
    \begin{align*}
        &\enspace\quad\mathcal{V}_{i+1}(K)=b_{n,i}\int_{(\langle K^{\frac{i+1}{n}}\rangle)^{i+1}}\Lambda(x_1,\ldots,x_{i+1})\mathsmaller{\prod}\limits_{j=1}^{i+1}(|x_j|^{1-\frac{n}{i+1}})\thinspace dx_1\cdots dx_{i+1}
        \\
        &=b_{n,i}\int_{(P_{u^{\perp}}\langle K^{\frac{i+1}{n}}\rangle)^{i+1}}\int_{R_{\boldsymbol{y}}(\langle K^{\frac{i+1}{n}}\rangle)} \Lambda_{u,\boldsymbol{y}}(s_1,\ldots,s_{i+1})\mathsmaller{\prod}\limits_{j=1}^{i+1}|y_j+s_ju|^{1-\frac{n}{i+1}}\thinspace ds_1\cdots ds_{i+1} d\boldsymbol{y}
        \\
        &=b_{n,i}\int_{(P_{u^{\perp}}\langle K^{\frac{i+1}{n}}\rangle)^{i+1}}\int_{R_{\boldsymbol{y}}(\langle K^{\frac{i+1}{n}}\rangle)}\Lambda_{u,\boldsymbol{y}}(s_1,\ldots,s_{i+1})\big(\mathsmaller{\prod}\limits_{j=1}^{i+1}\varphi_{y_j}(s_j)\big)\thinspace ds_1\cdots ds_{i+1} d\boldsymbol{y}
        \\
        &=b_{n,i}\int_{(P_{u^{\perp}}\langle K^{\frac{i+1}{n}}\rangle)^{i+1}}\int_{R_{\boldsymbol{y}}(\langle K^{\frac{i+1}{n}}\rangle)}\big(\int_{\mathbb{R}_+}1_{[\Lambda_{u,\boldsymbol{y}}]_{\alpha}}\thinspace d\alpha\cdot\big(\mathsmaller{\prod}\limits_{j=1}^{i+1}\int_{\mathbb{R}_+}1_{[\varphi_{y_j}]_{z_j}}\thinspace dz_j\big)\big)\thinspace ds_1\cdots ds_{i+1} d\boldsymbol{y}
        \\
        &=b_{n,i}\int_{(P_{u^{\perp}}\langle K^{\frac{i+1}{n}}\rangle)^{i+1}}\int_{\mathbb{R}_+}\int_{(\mathbb{R}_+)^{i+1}}V_{i+1}\big(R_{\boldsymbol{y}}(\langle K^{\frac{i+1}{n}}\rangle)\cap [\Lambda_{u,\boldsymbol{y}}]_{\alpha}\cap H_{\boldsymbol{y}}(z)\big)\thinspace dz d\alpha d\boldsymbol{y}
        \\
        &=\mathcal{V}_{i+1,u}(K).
    \end{align*}

    Hence, $\widetilde{V}_{i+1}(I_iK)=\mathcal{V}_{i+1}(K)=\mathcal{V}_{i+1,u}(K)$ for $i\in\{1,2,\ldots,n-2\}$.
     \end{proof}

\begin{rem}
    It is emphasized that this theorem applies only to $i\in\{1,2,\ldots,n-2\}$, due to the condition $i+1\leq n-1$ arising from the equality $\eqref{formula}$, which distinguishes  the case $i\in\{1,2,\ldots,n-2\}$ from the case $i=n-1$.
\end{rem}
Specifically, if $K$ is a Lipschitz star body, then for $u\in U(\langle K^{\frac{i+1}{n}}\rangle)$, $\langle K^{\frac{i+1}{n}}\rangle$ is a $u$-multi-graphical star body. Let $\Omega_{\infty}$ be the open
subset of $P_{u^{\perp}}\langle K^{\frac{i+1}{n}}\rangle$ given by Definition $\ref{multi}$.  Recall that $\mathcal{H}^{n-1}(P_{u^{\perp}}\langle K^{\frac{i+1}{n}}\rangle \backslash \Omega_{\infty}) = 0$. By Theorem $\ref{reformulation}$, it follows that
\begin{align}
\label{V}
    &\widetilde{V}_{i+1}(I_iK_t)=\mathcal{V}_{i+1,u}(K_t) \nonumber
    \\
    =\enspace &b_{n,i}\int_{(\Omega_{\infty})^{i+1}}\int_{\mathbb{R}_+}\int_{(\mathbb{R}_+)^{i+1}}V_{i+1}\big(R_{\boldsymbol{y}}(S_u^t\langle K^{\frac{i+1}{n}}\rangle)\cap [\Lambda_{u,\boldsymbol{y}}]_{\alpha}\cap H_{\boldsymbol{y}}(z)\big)\thinspace dz d\alpha d\boldsymbol{y}.
\end{align}

Note that for $\boldsymbol{y}\in (\Omega_{\infty})^{i+1}$, by Lemma $\ref{Continuous Steiner symmetrization}$ $(2)$ we have
 \begin{align*}
 R_{\boldsymbol{y}}(S_u^t\langle K^{\frac{i+1}{n}}\rangle)&=\{(s_1,\ldots,s_{i+1}):y_j+s_ju\in S_u^t\langle K^{\frac{i+1}{n}}\rangle,\thinspace j=1,\ldots, i+1 \}
 \\
 &=\{(s_1,\ldots,s_{i+1}):y_j+s_ju\in \mathring{S}_u^t\langle K^{\frac{i+1}{n}}\rangle,\thinspace j=1, \ldots, i+1\}
 \\
 &=S^t(\langle K^{\frac{i+1}{n}}\rangle\cap L^{y_1}_u)\times\cdots\times S^t(\langle K^{\frac{i+1}{n}}\rangle\cap L^{y_{i+1}}_u).
 \end{align*}
 That is, $R_{\boldsymbol{y}}(S_u^t\langle K^{\frac{i+1}{n}}\rangle)$ is a finite disjoint union of rectangles in $\mathbb{R}^{i+1}$. Throughout this article, a rectangle refers
to a compact axis-aligned rectangle with non-empty interior.

\subsection{Proof of Theorem \ref{1.1}}\
\label{3.3}
\vskip 5pt

\begin{proof}[\bf{Proof of Theorem $\ref{periodic}$}\thinspace\rm : ]
    If $K$ is an origin-symmetric ball, then it is clear that $I_i^2K = cK$ for some $c > 0$. Conversely, assume that $K$ is a star body in $\r$ with $n \geq 3$ such that $I_i^2K = cK$ for some $c>0$.

    Since $\rho_{I_iK} = \frac{1}{n-1} \R(\rho^{i}_K)$, it follows that
    \begin{align*}
        c\rho_{K}=\rho_{I_i^2 K}=\frac{1}{(n-1)^{i+1}}\R((\R(\rho_K^i))^i).
    \end{align*}
    So, $\rho_K$ is $C^{\infty}$-smooth by Theorem $\ref{regularity}$, and therefore $\rho_K$ is Lipschitz continuous. Hence,  $K$ is a Lipschitz star body, and so is $\langle K^{\frac{i+1}{n}}\rangle$. By Theorem $\ref{almost direction}$,  there exists a Lebesgue measurable set $U(\langle K^{\frac{i+1}{n}}\rangle) \subseteq \s$ of full measure such that $\langle K^{\frac{i+1}{n}}\rangle$ is $u$-multi-graphical for all $u\in U(\langle K^{\frac{i+1}{n}}\rangle)$.  Let $\Omega_{\infty}$ be the open subset of $P_{u^{\perp}}\langle K^{\frac{i+1}{n}}\rangle$ given by Definition \ref{multi}.

    In the following, we divide the proof into five steps.

    $\textbf{Step 1.}$ Prove that for any origin-symmetric convex set $A\subseteq \mathbb{R}^{i+1}$ and any rectangle $Q\subseteq \mathbb{R}^{i+1}$ with centroid $c(Q)$,  $t\mapsto V_{i+1}(A\cap (Q-tc(Q)))$
    is \emph{increasing} in $t\in[0,1]$.

    Indeed, since $A$ and $Q-c(Q)$ are origin-symmetric, it follows that the function
    \begin{align*}
       \mathbb{R}\ni t\mapsto V_{i+1}(A\cap (Q-c(Q)-tc(Q)))^{\frac{1}{i+1}}
    \end{align*}
    is even.  Also, it is concave on its support  by Lemma \ref{Brunn}. So, it is increasing on $(-\infty,0]$ and is decreasing on $[0,\infty)$. Thus,  $ t\mapsto V_{i+1}(A\cap (Q-tc(Q)))$ is increasing in $t\in[0,1]$.

    $\textbf{Step 2.}$ Let $u\in U(\langle K^{\frac{i+1}{n}}\rangle)$. Prove that for every $\alpha>0$, $\boldsymbol{y}\in (\Omega_{\infty})^{i+1}$  and $z\in (\mathbb{R}_+)^{i+1}$,
    \begin{align*}
         \phi (t; \alpha, \boldsymbol{y},  z):= V_{i+1}\big(R_{\boldsymbol{y}}(S_u^t\langle K^{\frac{i+1}{n}}\rangle)\cap [\Lambda_{u,\boldsymbol{y}}]_{\alpha}\cap H_{\boldsymbol{y}}(z)\big)
    \end{align*}
     is \emph{increasing} in $t\in[0,1]$. If so, for $\{K_t = \langle (S^t_{u} \langle K^{\frac{i+1}{n}}\rangle)^{\frac{n}{i+1}}\rangle\}_{t\in[0,1]}$, by Lemma \ref{again admissible}, together with the equality \eqref{V} and the Fatou lemma, we have
     \begin{align*}
        0=\frac{d\widetilde{V}_{i+1}(I_iK_t)}{dt}\Big|_{0^+}\geq b_{n,i}\int_{(\Omega_{\infty})^{i+1}}\int_{\mathbb{R}_+}\int_{(\mathbb{R}_+)^{i+1}}\underline{\frac{d}{dt}}\phi (t; \alpha, \boldsymbol{y},  z)\Big|_{0^+}\thinspace dz d\alpha d\boldsymbol{y}
       \geq 0,
    \end{align*}
    which implies that  for a.e. $\boldsymbol{y}\in (\Omega_{\infty})^{i+1}$, we have
    \begin{align*}
    \int_{\mathbb{R}_+}\int_{(\mathbb{R}_+)^{i+1}}\underline{\frac{d}{dt}}\phi (t; \alpha, \boldsymbol{y},  z)\Big|_{0^+}\thinspace dz d\alpha = 0.
    \end{align*}

     Indeed, for each $t \in [0, 1]$, recall that $ R_{\boldsymbol{y}}(S_u^t \langle K^{\frac{i+1}{n}}\rangle)$ is the disjoint union of finitely many rectangles
    $Q_k^t$ with centroid $c(Q_k^t)$. Let $0 = \tau_0 < \tau_1 < \cdots < \tau_N = 1$ denote the
    collision times of the $B_k^t$'s as they evolve in time. For each $j\in\{0,1,\ldots,N-1\}$,
    we verify that $\phi (t; \alpha, \boldsymbol{y},  z)$ is increasing in $t \in [\tau_j ,\tau_{j+1}]$.

    Note that for $t \in [\tau_j ,\tau_{j+1})$, each $Q_k^{\tau_j}$ evolves independently as $Q_k^t = Q_k^{\tau_j}-\frac{t-\tau_j}{1-\tau_j} c(Q_k^{\tau_j} )$. So,
    \begin{align*}
        \phi(t; \alpha, \boldsymbol{y},  z)=\mathsmaller{\sum}\limits_k\thinspace V_{i+1}\big((Q_k^{\tau_j}-\frac{t-\tau_j}{1-\tau_j} c(Q_k^{\tau_j} ))\cap [\Lambda_{u,\boldsymbol{y}}]_{\alpha}\cap H_{\boldsymbol{y}}(z)\big).
    \end{align*}
    Recall that $H_{\boldsymbol{y}}(z)=[\varphi_{y_1}]_{z_1}\times[\varphi_{y_2}]_{z_2}\times\cdots\times[\varphi_{y_{i+1}}]_{z_{i+1}}$ and $\varphi_{y_j}(t)=|y_j+tu|^{1-\frac{n}{i+1}}$, $ t\in\mathbb{R}$ for $j=1,2,\ldots,i+1$. Since $\varphi_{y_j}$ is even in $\mathbb{R}$ and is strictly decreasing in $[0,\infty)$,  it follows that $H_{\boldsymbol{y}}(z)$ is an origin-symmetric rectangle in $\mathbb{R}^{i+1}$.
    Combining the fact that $[\Lambda_{u,\boldsymbol{y}}]_{\alpha}$ is an origin-symmetric convex set by Lemma $\ref{convex}$,    it yields that $\phi (t; \alpha, \boldsymbol{y},  z)$ is increasing in $t \in [\tau_j ,\tau_{j+1}]$ by \mbox{Step 1}. Therefore, $\phi (t; \alpha, \boldsymbol{y},  z)$ is increasing in $t \in [0,1]$.

    $\textbf{Step 3.}$ Prove that for each given $\boldsymbol{y}\in (\Omega_{\infty})^{i+1}$, if $R_{\boldsymbol{y}}(\langle K^{\frac{i+1}{n}}\rangle)$ is \emph{not} an origin-symmetric rectangle in $\mathbb{R}^{i+1}$, then
    \begin{align*}
        \int_{\mathbb{R}_+}\int_{(\mathbb{R}_+)^{i+1}}\underline{\frac{d}{dt}}\phi (t; \alpha, \boldsymbol{y},  z)\Big|_{0^+}\thinspace  dzd\alpha > 0.
    \end{align*}
    If so, from Step 2, it follows that $R_{\boldsymbol{y}}(\langle K^{\frac{i+1}{n}}\rangle)$ is an origin-symmetric rectangle for a.e. $\boldsymbol{y}\in (\Omega_{\infty})^{i+1}$. In light of that $\boldsymbol{y}\mapsto R_{\boldsymbol{y}}(\langle K^{\frac{i+1}{n}}\rangle)$ is continuous in $(\Omega_{\infty})^{i+1}$ by Lemma $\ref{Continuous Steiner symmetrization}$ $(1)$, it follows that $R_{\boldsymbol{y}}(\langle K^{\frac{i+1}{n}}\rangle)$ is an origin-symmetric \mbox{rectangle} for all $\boldsymbol{y}\in (\Omega_{\infty})^{i+1}$. Consequently, for each $y\in \Omega_{\infty}$, $\langle K^{\frac{i+1}{n}}\rangle\cap L^y_u$ is a line segment symmetric with respect to $u^{\perp}$. Added with Lemma \ref{Continuous Steiner symmetrization} $(5)$, it follows that $\langle K^{\frac{i+1}{n}}\rangle=S_{u}^{0} \langle K^{\frac{i+1}{n}}\rangle =\operatorname{cl}\big(\mathring{S}_{u}^{0} K\big)$ is symmetric with respect to $u^{\perp}$ for all $u\in U(\langle K^{\frac{i+1}{n}}\rangle)$.

    Indeed, $R_{\boldsymbol{y}}(\langle K^{\frac{i+1}{n}}\rangle)$ is the disjoint union of finitely many rectangles
    $Q_k^0$. So,
    \begin{align*}
        \underline{\frac{d}{dt}}\phi (t; \alpha, \boldsymbol{y},  z)\Big|_{0^+}\geq \mathsmaller{\sum}\limits_k \thinspace\underline{\frac{d}{dt}}V_{i+1}((Q_k^0-tc(Q_k^0))\cap [\Lambda_{u,\boldsymbol{y}}]_{\alpha}\cap H_{\boldsymbol{y}}(z))\Big|_{0^+}.
    \end{align*}
    By Step 2, we know that $V_{i+1}((Q_k^0-tc(Q_k^0))\cap [\Lambda_{u,\boldsymbol{y}}]_{\alpha}\cap H_{\boldsymbol{y}}(z))$ is increasing in $t\in[0,1]$. So, each term involved in the summation above is non-negative. Thus, it suffices to show that there exist an index $k_0$, an open set  $O_1\subseteq (\mathbb{R}_+)^{i+1}$ and an open set $O_2\subseteq \mathbb{R}_+$ such that for any  $z\in O_1$ and $\alpha\in O_2$,
    \begin{align*}
        \underline{\frac{d}{dt}}V_{i+1}((Q_{k_0}^0-tc(Q_{k_0}^0))\cap [\Lambda_{u,\boldsymbol{y}}]_{\alpha}\cap H_{\boldsymbol{y}}(z))\Big|_{0^+}>0.
    \end{align*}

    By the assumption that $R_{\boldsymbol{y}}(\langle K^{\frac{i+1}{n}}\rangle)$ is not an origin-symmetric rectangle,  there exists an index $k_0$ such that $Q_{k_0}^0$ is not origin-symmetric. For convenience, write
    \begin{align*}
        Q_{k_0}^0=\mathsmaller{\prod}\limits_{j=1}^{i+1}[c_j-l_j,c_j+l_j],\quad l_j> 0.
    \end{align*}
    Then  there exists a $j\in\{1,2,\ldots,i+1\}$ so that $c_j\neq 0$. Say, $c_1\neq  0$. Assume $c_1>0$.

    Let
    \begin{align*}
        O_1=\{(z_1,z_2,\ldots,z_{i+1}): \enspace &\varphi_{y_1}(c_1+l_1)<z_1<\varphi_{y_1}(|c_1-l_1|),
        \\
        &\varphi_{y_j}(2M)<z_j<\varphi_{y_j}(M),\enspace  j=2,3,\ldots,i+1\},
    \end{align*}
    where $M\geq 2(|c_j|+l_j)$, for $j=1,2\ldots,i+1$. Then $O_1$ is an open subset of $(\mathbb{R}_+)^{i+1}$.

    For any $z\in O_1$, write the rectangle $H_{\boldsymbol{y}}(z)=[\varphi_{y_1}]_{z_1}\times[\varphi_{y_2}]_{z_2}\times\cdots\times[\varphi_{y_{i+1}}]_{z_{i+1}}$ as $[-h_1,h_1]\times [-h_2,h_2]\times\cdots\times [-h_{i+1},h_{i+1}]$. In light of that $\varphi_{y_j}=|y_j+tu|^{1-\frac{n}{i+1}}$, $ t\in\mathbb{R}$, is even in $\mathbb{R}$ and is strictly decreasing in $[0,\infty)$, it follows that
    \begin{align*}
        &|c_1-l_1|<h_1< c_1+l_1,\quad\text{and}\quad M< h_{j}< 2M,\quad j=2,3,\ldots,i+1.
    \end{align*}

    In particular, $c_1-l_1<h_1< c_1+l_1$ and $l_1-c_1<h_1$, i.e., $-h_1< c_1-l_1<h_1<c_1+l_1$. These chosen $h_1,\ldots,h_{i+1}$ ensure that the intersection $Q_{k_0}^0\cap H_{\boldsymbol{y}}(z)$ meets the boundary of $H_{\boldsymbol{y}}(z)$ \emph{solely} on the facet $\{h_1\}\times [-h_2,h_2]\times\cdots\times [-h_{i+1},h_{i+1}]$.  See the following \mbox{Figure 1.}

    \begin{figure}[ht]
\centering
\begin{tikzpicture}[scale=1.05]
\definecolor{wde}{RGB}{21, 102, 192}
    % 定义关键坐标
    \def\h{1.5}    % h1 (垂直方向半高)
    \def\c{1.35}      % c1 (垂直方向中心)
    \def\l{1}    % l1 (垂直方向半宽)
    \def\H{5}    % 水平方向半宽 (h2)
    \def\L{0.95}    % 水平方向半宽 (l2)

    \pgfmathsetmacro{\interbottom}{\c-\l}
    \pgfmathsetmacro{\intertop}{\h}
    \pgfmathsetmacro{\interleft}{1.2-\L}
    \pgfmathsetmacro{\interright}{1.2+\L}

    \draw[wde,  fill=wde!10] (-\H,-\h) rectangle (\H,\h);
     \node[black, anchor=south west, xshift=0pt, yshift=-22pt, scale=1.1] at (-\H, \h) {$H_{\boldsymbol{y}}(z)$};

    \draw[wde, thick, fill=wde!6] (1.2-\L, \c-\l) rectangle (1.2+\L, \c+\l);
    \node[black, anchor=south west, xshift=0pt, yshift=-22pt, scale=1] at (1.2-\L, \c+\l) {$Q_{k_0}^0$};
    \filldraw (1.2,\c-0.1) circle (1pt) node[black, scale=0.8 ,below right, inner sep=0pt, xshift=2pt, yshift=0pt] {$c(Q_{k_0}^0)$};
    \draw[wde, very thick] (1.2-\L, \h) -- (1.2+\L, \h);
    \fill[wde, opacity=0.2] (\interleft, \interbottom) rectangle (\interright, \intertop);

    \draw[dashed, wde] (-\H, \c-\l) -- (\H, \c-\l);
    \draw[dashed, wde] (-\H, \c+\l) -- (\H, \c+\l);

    \node[right, black, font=\small] at (\H, -\h) {$-h_1$};
    \node[right, black, font=\small] at (\H, \h) {$h_1$};
    \node[right,black, font=\small] at (\H, \c-\l) {$c_1-l_1$};
    \node[right, black, font=\small] at (\H, \c+\l) {$c_1+l_1$};

    \filldraw (0,0) circle (1pt) node[inner sep=0pt, xshift=2pt, yshift=-1pt,scale=1 ,below right] {$o$};
\end{tikzpicture}
\caption{Rectangles $H_{\boldsymbol{y}}(z)$ and $Q_{k_0}^0$.}
\end{figure}

%%
\iffalse
     \begin{figure}[H]
\centering
 \includegraphics[width=0.7\linewidth]{rectangle.png}
\caption{Rectangles $H_{\boldsymbol{y}}(z)$ and $Q_{k_0}^0$.}
\end{figure}
\fi

   Recall that $[\Lambda_{u,\boldsymbol{y}}]_{\alpha}=\{(s_1,\ldots,s_{i+1}):V_{i+1}(\conv\{o,y_1+s_1u,\ldots,y_{i+1}+s_{i+1}u\})\leq \frac{1}{\alpha}\}$, $\alpha\in[0,\infty].$
    So, $\lim\limits_{\alpha\to 0^+}[\Lambda_{u,\boldsymbol{y}}]_{\alpha}=\mathbb{R}^{i+1}$, and there exists an open set $O_2\subseteq\mathbb{R}_+$ such that
    \begin{align*}
[\Lambda_{u,\boldsymbol{y}}]_{\alpha}\supseteq[-10M,10M]^{i+1},\quad \forall\thinspace \alpha\in O_2.
    \end{align*}

    Consequently, for any  $z\in O_1$, $\alpha\in O_2$ and $t\in[0,1]$, we have
    \begin{align*}
        V_{i+1}\big((Q_{k_0}^0-tc(Q_{k_0}^0))\cap [\Lambda_{u,\boldsymbol{y}}]_{\alpha}\cap H_{\boldsymbol{y}}(z)\big)=V_{i+1}\big((Q_{k_0}^0-tc(Q_{k_0}^0))\cap H_{\boldsymbol{y}}(z)\big),
    \end{align*}
    and therefore
    \begin{align*}
        &\quad\enspace\underline{\frac{d}{dt}}V_{i+1}((Q_{k_0}^0-tc(Q_{k_0}^0))\cap [\Lambda_{u,\boldsymbol{y}}]_{\alpha}\cap H_{\boldsymbol{y}}(z))\Big|_{0^+}=\underline{\frac{d}{dt}}V_{i+1}((Q_{k_0}^0-tc(Q_{k_0}^0))\cap H_{\boldsymbol{y}}(z))\Big|_{0^+}
        \\
        &=\frac{d}{dt}V_{i+1}((Q_{k_0}^0-tc(Q_{k_0}^0))\cap H_{\boldsymbol{y}}(z))\Big|_{0^+}=c_1l_2\cdots l_{i+1}>0.
    \end{align*}

    $\textbf{Step 4.}$ Prove that $\langle K^{\frac{i+1}{n}}\rangle$ is an origin-symmetric ball, and therefore $K$ is an origin-symmetric ball as desired.

    Indeed, fix \mbox{$u_0\in \s$} and let $r_u(x)=x-2(x\cdot u)u$, $x\in\r$ for $u\in\s$. From Step 3, we note that for all $u\in U(\langle K^{\frac{i+1}{n}}\rangle)$, $\langle K^{\frac{i+1}{n}}\rangle$ is symmetric with respect to $u^{\perp}$, i.e., $\langle K^{\frac{i+1}{n}}\rangle=r_u(\langle K^{\frac{i+1}{n}}\rangle)$. Thus,
    \begin{align*}
        \rho_{\langle K^{\frac{i+1}{n}}\rangle}(r_{u}(u_0))=\rho_{r_u(\langle K^{\frac{i+1}{n}}\rangle)}(r_{u}(u_0))=\rho_{\langle K^{\frac{i+1}{n}}\rangle}(u_0),\quad \forall\thinspace u\in  U(\langle K^{\frac{i+1}{n}}\rangle).
    \end{align*}
 In light of that \mbox{$\mathcal{H}^{n-1}(\s\backslash U(\langle K^{\frac{i+1}{n}}\rangle))=0$} and $\rho_{\langle K^{\frac{i+1}{n}}\rangle}$ is continuous on $\s$, we have
    \begin{align*}
        \rho_{\langle K^{\frac{i+1}{n}}\rangle}(r_{u}(u_0))=\rho_{\langle K^{\frac{i+1}{n}}\rangle}(u_0), \quad\forall \thinspace u\in\s.
    \end{align*}
    Thus, $\rho_{\langle K^{\frac{i+1}{n}}\rangle}$ is constant on $\s$, i.e., $\langle K^{\frac{i+1}{n}}\rangle$ is an origin-symmetric ball.

    Combining the above four steps, we finish the proof that $I_i^2K = cK$ for some $c > 0$ if and only if $K$ is an origin-symmetric ball.

    $\textbf{Step 5.}$ Prove that  $I_iK = cK$ for some $c > 0$ iff $K$ is an origin-symmetric ball.

    If $K = B_n(r)$, then it is clear that
\begin{align*}
    I_iK = r^{i}I_iB_n = r^{i}\omega_{n-1}B_n =r^{i-1}\omega_{n-1}K.
\end{align*}
Conversely, if $I_iK = cK$ for some $c > 0$, then $I_i^2K = I_i(cK) = c^{i+1}K$, and therefore $K$ is an origin-symmetric ball.
\end{proof}

\begin{rem}
    Theorem \ref{periodic}  remains valid when $K$ is a star-shaped bounded Borel set such that $I_i^2 K = cK$ or $I_iK = cK$ holds up to an $\mathcal{H}^{n}$-null set. In such cases,  $\rho_K(u)\in L^{\infty}(\s)$ and  $\R(\R(\rho_K^i)^i)=\tilde{c}\rho_K$ holds up to an $\mathcal{H}^{n-1}$-null set for some $\tilde{c}>0$. Then, by Theorem $\ref{regularity}$, after modifying $\rho_K$ on an $\mathcal{H}^{n-1}$-null set (and thus modifying $K$ on an $\mathcal{H}^n$-null set), $\rho_K\in C^{\infty}(\s)$ is either identically zero or strictly positive. Therefore, the modified $K$ is a Lipschitz star body, and the proof proceeds as above.
\end{rem}

\section{\bf Generalized Busemann intersection inequalities}
\label{4}
\vskip 10pt

In this part, we finish the proof of Theorem \ref{isoperimetric}. To prove Theorem \ref{isoperimetric}, we use the following theorem, which is precisely Theorem 5.1 in \cite{Lin2}.

\begin{thm}
\label{convergence}
    Let $K$ be a Lipschitz star body in $\r$ and $U(K)$ be the set given by Theorem \ref{almost direction} applied to $K$. Then there exists $\{u_j\}_{j=1}^{\infty}\subseteq \s$ such that $u_1\in U(K)$, $u_{j+1}\in U(K_j)$ for $j\geq 1$ and  $\max\limits_{u\in\s}|\rho_{K_j}(u)-\rho_{B_K}(u)|\to 0$ as $ j\to\infty,$ where $K_j=S^1_{u_j}\cdots S^1_{u_1} K$ and $B_K$ denotes the origin-symmetric ball having the same volume as $K$.
\end{thm}

\begin{proof}[\bf{Proof of Theorem $\ref{isoperimetric}$}\thinspace\rm : ] For any star body $K$ in $\r$, let
\begin{align*}
    \mathcal{D}_{i+1}(K)=b_{n,i}\int_{(\r)^{i+1}}\frac{\prod_{j=1}^{i+1}(|x_j|^{1-\frac{n}{i+1}}1_{K}(x_j))}{V_{i+1}(\conv\{o,x_1,\ldots,x_{i+1}\})}\thinspace dx_1\cdots dx_{i+1}.
\end{align*}
From Theorem \ref{3.2}, it follows  that $\widetilde{V}_{i+1}(I_i K)=\mathcal{V}_{i+1}(K)=\mathcal{D}_{i+1}(\langle K^{\frac{i+1}{n}}\rangle)$. Together with the identity $\widetilde{V}_{i+1}(K)=V_n(\langle K^{\frac{i+1}{n}}\rangle)$, we have
 \begin{align*}
     \frac{\widetilde{V}_{i+1}(I_i K)}{\widetilde{V}_{i+1}(K)^i}=\frac{\mathcal{D}_{i+1}(\langle K^{\frac{i+1}{n}}\rangle)}{V_n(\langle K^{\frac{i+1}{n}}\rangle)^i}\quad\text{and}\quad \frac{\widetilde{V}_{i+1}(I_i B_n)}{\widetilde{V}_{i+1}(B_n)^i}=\frac{\mathcal{D}_{i+1}(\langle B_n^{\frac{i+1}{n}}\rangle)}{V_n(\langle B_n^{\frac{i+1}{n}}\rangle)^i}=\frac{\mathcal{D}_{i+1}(B_{\langle K^{\frac{i+1}{n}}\rangle})}{V_n(B_{\langle K^{\frac{i+1}{n}}\rangle})^i},
 \end{align*}
 where $B_{\langle K^{\frac{i+1}{n}}\rangle}$ is the origin-symmetric ball having the same volume as $\langle K^{\frac{i+1}{n}}\rangle$.

 Hence, for a Lipschitz star body $K$, to prove the inequality
\begin{align*}
    \frac{\widetilde{V}_{i+1}(I_i K)}{\widetilde{V}_{i+1}(K)^i}\leq \frac{\widetilde{V}_{i+1}(I_i B_n)}{\widetilde{V}_{i+1}(B_n)^i},
\end{align*}
  it suffices to prove the inequality $\mathcal{D}_{i+1}(\langle K^{\frac{i+1}{n}}\rangle)\leq \mathcal{D}_{i+1}(B_{\langle K^{\frac{i+1}{n}}\rangle}).$

     For $\{K_t = \langle (S^t_{u} \langle K^{\frac{i+1}{n}}\rangle)^{\frac{n}{i+1}}\rangle\}_{t\in[0,1]}$ with $u\in U(\langle K^{\frac{i+1}{n}}\rangle)$,  $\mathcal{D}_{i+1}(S_u^t\langle K^{\frac{i+1}{n}}\rangle)=\widetilde{V}_{i+1}(I_i K_t)$ is increasing in $t\in[0,1]$ as shown in Step 2 in the proof of Theorem \ref{periodic}. Thus, by Lemma \ref{Continuous Steiner symmetrization} $(5)$, it follows that
     \begin{align*}
         \mathcal{D}_{i+1}(\langle K^{\frac{i+1}{n}}\rangle)= \mathcal{D}_{i+1}(S_u^0\langle K^{\frac{i+1}{n}}\rangle)\leq \mathcal{D}_{i+1}(S_u^1\langle K^{\frac{i+1}{n}}\rangle),\quad\forall\thinspace u\in U(\langle K^{\frac{i+1}{n}}\rangle).
     \end{align*}
      Let $\{u_j\}_{j=1}^{\infty}$ be the sequence given by Theorem \ref{convergence} applied to $\langle K^{\frac{i+1}{n}}\rangle$. It follows that
     \begin{align*}
        \mathcal{D}_{i+1}(\langle K^{\frac{i+1}{n}}\rangle)\leq \lim_{j\to \infty} \mathcal{D}_{i+1}(S^1_{u_j}\cdots S^1_{u_1}\langle K^{\frac{i+1}{n}}\rangle)=\mathcal{D}_{i+1}(B_{\langle K^{\frac{i+1}{n}}\rangle}).
    \end{align*}

    Now, we prove equality conditions. It suffices to show that $\mathcal{D}_{i+1}(\langle K^{\frac{i+1}{n}}\rangle)= \mathcal{D}_{i+1}(B_{\langle K^{\frac{i+1}{n}}\rangle})$ if and only if $K$ is an origin-symmetric ball. If $K$ is an origin-symmetric ball, then it is clear that $\langle K^{\frac{i+1}{n}}\rangle=B_{\langle K^{\frac{i+1}{n}}\rangle}$ and the equality holds.

    Otherwise, assume that $K$ is not an origin-symmetric ball. We claim that there exists $u_o\in U(\langle K^{\frac{i+1}{n}}\rangle)$ so that $\underline{\frac{d}{dt}}\mathcal{D}_{i+1}(S^t_{u_0} \langle K^{\frac{i+1}{n}}\rangle)\big|_{0^+}>0.$
    If so, then there is $t_0\in (0,1]$ so that
    \begin{align*}
        \mathcal{D}_{i+1}(\langle K^{\frac{i+1}{n}}\rangle)<  \mathcal{D}_{i+1}(S^{t_0}_{u_0}\langle K^{\frac{i+1}{n}}\rangle)\leq\mathcal{D}_{i+1}(B_{\langle K^{\frac{i+1}{n}}\rangle}),
    \end{align*}
    which yields that $\mathcal{D}_{i+1}(\langle K^{\frac{i+1}{n}}\rangle)\neq \mathcal{D}_{i+1}(B_{\langle K^{\frac{i+1}{n}}\rangle})$.

    Indeed, for $u\in U(\langle K^{\frac{i+1}{n}}\rangle)$,  since $\mathcal{D}_{i+1}(S_u^t\langle K^{\frac{i+1}{n}}\rangle)$ is increasing in $t\in[0,1]$, it follows that $\underline{\frac{d}{dt}}\mathcal{D}_{i+1}(S_u^t\langle K^{\frac{i+1}{n}}\rangle)\big|_{0^+}\geq0$. If such $u_0$ does not exist, then
    \begin{align*}
        \underline{\frac{d}{dt}}\mathcal{V}_{i+1,u}(K_t)\Big|_{0^+}=\underline{\frac{d}{dt}}\mathcal{D}_{i+1}(S_u^t\langle K^{\frac{i+1}{n}}\rangle)\Big|_{0^+}=0, \quad \forall\thinspace u\in U(\langle K^{\frac{i+1}{n}}\rangle).
    \end{align*}
     Repeating the process of Steps 3 and 4 in the proof of Theorem \ref{periodic}, we conclude that $K$ is an origin-symmetric ball, which contradicts the assumption.
    \end{proof}

\end{CJK*}

\end{document}